\documentclass[12pt]{amsart}
\usepackage{amsmath,amsthm,amsfonts,amssymb,mathrsfs}
\date{\today}

\usepackage{color}

\newtheorem{theorem}{Theorem}[section]
\newtheorem{proposition}[theorem]{Proposition}%[section]
\newtheorem{corollary}[theorem]{Corollary}%[section]
\newtheorem{lemma}[theorem]{Lemma}

\theoremstyle{definition}
\newtheorem{example}[theorem]{Example}%[section]
\newtheorem{remark}[theorem]{Remark}%[section]
\usepackage{hyperref}

\begin{document}

\title[Topological semigroups]{Topological monoids of
monotone injective partial selfmaps of $\mathbb{N}$ with
cofinite domain and image}

\author[O.~Gutik]{Oleg~Gutik}
\address{Department of Mathematics, Ivan Franko Lviv National
University, Universytetska 1, Lviv, 79000, Ukraine}
\email{o\underline{\hskip5pt}\,gutik@franko.lviv.ua,
ovgutik@yahoo.com}

\author[D.~Repov\v{s}]{Du\v{s}an~Repov\v{s}}
\address{Faculty of Mathematics and Physics, and
Faculty of Education, University of Ljubljana, P.~O.~B. 2964,
Ljubljana, 1001, Slovenia} \email{dusan.repovs@guest.arnes.si}

\keywords{Topological semigroup, topological inverse semigroup,
bicyclic semigroup, semigroup of bijective partial
transformations, closure}

\subjclass[2010]{Primary 22A15, 20M20. Secondary 20M18, 54H15}

\begin{abstract}
In this paper we study the semigroup
$\mathscr{I}_{\infty}^{\!\nearrow}(\mathbb{N})$ of partial cofinal
monotone bijective transformations of the set of positive integers
$\mathbb{N}$. We show that the semigroup
$\mathscr{I}_{\infty}^{\!\nearrow}(\mathbb{N})$ has algebraic
properties similar to the bicyclic semigroup: it is bisimple and
all of its non-trivial group homomorphisms are either isomorphisms
or group homomorphisms. We also prove that every locally compact
topology $\tau$ on $\mathscr{I}_{\infty}^{\!\nearrow}(\mathbb{N})$
such that $(\mathscr{I}_{\infty}^{\!\nearrow}(\mathbb{N}),\tau)$
is a topological inverse semigroup, is discrete. Finally, we describe the
closure of $(\mathscr{I}_{\infty}^{\!\nearrow}(\mathbb{N}),\tau)$
in a topological semigroup.
\end{abstract}

\maketitle

%\tableofcontents

\section{Introduction and preliminaries}

Our purpose is
to study the semigroup
$\mathscr{I}_{\infty}^{\!\nearrow}(\mathbb{N})$ of partial cofinal
monotone bijective transformations of the set of positive integers
$\mathbb{N}$. We shall
show that the semigroup
$\mathscr{I}_{\infty}^{\!\nearrow}(\mathbb{N})$ has algebraic
properties similar to the bicyclic semigroup: it is bisimple and
all of its nontrivial group homomorphisms are either isomorphisms
or group homomorphisms.
We shall
also prove that every locally compact
topology $\tau$ on $\mathscr{I}_{\infty}^{\!\nearrow}(\mathbb{N})$
such that $(\mathscr{I}_{\infty}^{\!\nearrow}(\mathbb{N}),\tau)$
is a topological inverse semigroup is discrete and we
shall
describe the
closure of $(\mathscr{I}_{\infty}^{\!\nearrow}(\mathbb{N}),\tau)$
in a topological semigroup.

In this paper all spaces will be
assumed to be Hausdorff. Furthermore
we shall follow the terminology of \cite{CHK, CP, Engelking1989}.
 We shall denote the first infinite cardinal by $\omega$
 and  the cardinality of the set $A$
 by
$|A|$. If $Y$ is a subspace of a
topological space $X$ and $A\subseteq Y$, then  we shall
denote the topological closure and the interior of $A$ in $Y$
by
$\operatorname{cl}_Y(A)$
and
$\operatorname{Int}_Y(A)$,
respectively.

An algebraic semigroup $S$ is called {\it inverse} if for any
element $x\in S$ there exists a
unique $x^{-1}\in S$ such that
$xx^{-1}x=x$ and $x^{-1}xx^{-1}=x^{-1}$. The element $x^{-1}$ is
called the {\it inverse of} $x\in S$. If $S$ is an inverse
semigroup, then the function $\operatorname{inv}\colon S\to S$
which assigns to every element $x$ of $S$ its inverse element
$x^{-1}$ is called an {\it inversion}.

If $S$ is a semigroup, then we shall denote the
\emph{band} (i.~e. the subset of idempotents) of $S$
by $E(S)$. If the band
$E(S)$ is a nonempty subset of $S$, then the semigroup operation
on $S$ determines the partial order $\leqslant$ on $E(S)$:
$e\leqslant f$ if and only if $ef=fe=e$. This order is called {\em
natural}.

A \emph{semilattice} is a commutative semigroup of
idempotents. A semilattice $E$ is called {\em linearly ordered} or
\emph{chain} if the semilattice operation admits a linear natural
order on $E$. A \emph{maximal chain} of a semilattice $E$ is a
chain which is properly contained in no other chain of $E$. The
Axiom of Choice implies the existence of maximal chains in any
partially ordered set. According to
\cite[Definition~II.5.12]{Petrich1984} a
chain $L$ is called
an
$\omega$-chain if $L$ is isomorphic to $\{0,-1,-2,-3,\ldots\}$
with the usual order $\leqslant$. Let $E$ be a semilattice and
$e\in E$. We denote ${\downarrow} e=\{ f\in E\mid f\leqslant e\}$
and ${\uparrow} e=\{ f\in E\mid e\leqslant f\}$.

A {\it topological} ({\it inverse}) {\it semigroup} is a
topological space together with a continuous multiplication (and
an inversion, respectively).
Let $\mathscr{I}_\lambda$ denote the set of all partial one-to-one
transformations of a set $X$ of cardinality $\lambda$ together
with the following semigroup operation:
$x(\alpha\beta)=(x\alpha)\beta$ if
$x\in\operatorname{dom}(\alpha\beta)=\{
y\in\operatorname{dom}\alpha\mid
y\alpha\in\operatorname{dom}\beta\}$,  for
$\alpha,\beta\in\mathscr{I}_\lambda$. The semigroup
$\mathscr{I}_\lambda$ is called the \emph{symmetric inverse
semigroup} over the set $X$~(see \cite{CP}). The symmetric inverse
semigroup was introduced by Wagner~\cite{Wagner1952} and it plays
a major role in the theory of semigroups.

Let $\mathbb{N}$ be the set of all positive integers. We shall denote
the semigroup of monotone, non-decreasing, injective partial
transformations of $\mathbb{N}$ such that the sets
$\mathbb{N}\setminus\operatorname{dom}\varphi$ and
$\mathbb{N}\setminus\operatorname{rank}\varphi$ are finite for all
$\varphi\in\mathscr{I}_{\infty}^{\!\nearrow}(\mathbb{N})$
by
$\mathscr{I}_{\infty}^{\!\nearrow}(\mathbb{N})$ .
Obviously, $\mathscr{I}_{\infty}^{\!\nearrow}(\mathbb{N})$ is an
inverse subsemigroup of the semigroup $\mathscr{I}_\omega$. The
semigroup $\mathscr{I}_{\infty}^{\!\nearrow}(\mathbb{N})$ is
called \emph{the semigroup of cofinite monotone partial
bijections} of $\mathbb{N}$.

We shall denote every element $\alpha$ of the semigroup
$\mathscr{I}_{\infty}^{\!\nearrow}(\mathbb{N})$ by
 $
\left(%
\begin{array}{ccccc}
  n_1 & n_2 & n_3 & n_4 & \ldots \\
  m_1 & m_2 & m_3 & m_4 & \ldots \\
\end{array}%
\right)
 $
 and this means that $\alpha$ maps the
positive integer $n_i$ into $m_i$ for all $i=1,2,3,\ldots$~. In
this case the following conditions hold:
\begin{itemize}
    \item[$(i)$] the sets
    $\mathbb{N}\setminus\{n_1,n_2,n_3,n_4,\ldots\}$
    and $\mathbb{N}\setminus\{m_1, m_2, m_3, m_4, \ldots\}$ are
    finite;  and
    \item[$(ii)$] $n_1<n_2<n_3<n_4<\ldots$ and
    $m_1<m_2<m_3<m_4<\ldots$~.
\end{itemize}
We observe that an element $\alpha$ of the semigroup
$\mathscr{I}_{\omega}$ is an element of the semigroup
$\mathscr{I}_{\infty}^{\!\nearrow}(\mathbb{N})$ if and only if it
satisfies the conditions $(i)$ and $(ii)$.

The bicyclic semigroup ${\mathscr{C}}(p,q)$ is the semigroup with
the identity $1$,
generated by elements $p$ and $q$, subject only to
the condition $pq=1$. The bicyclic semigroup is bisimple and every
one of its congruences is either trivial or a group congruence.
Moreover, every non-annihilating homomorphism $h$ of the bicyclic
semigroup is either an isomorphism or the image of
${\mathscr{C}}(p,q)$ under $h$ is a cyclic group~(see
\cite[Corollary~1.32]{CP}). The bicyclic semigroup plays an
important role in algebraic theory of semigroups and in the theory
of topological semigroups.

For example,
the well-known result of
Andersen~\cite{Andersen} states that a ($0$--)simple semigroup is
completely ($0$--)simple if and only if it does not contain the
bicyclic semigroup. The bicyclic semigroup admits only the
discrete topology and a topological semigroup $S$ can contain
${\mathscr C}(p,q)$ only as an open
subset~\cite{EberhartSelden1969}. Neither stable nor
$\Gamma$-compact topological semigroups can contain a copy of the
bicyclic semigroup~\cite{AHK, HildebrantKoch1988}. Also, the
bicyclic semigroup does not embed into a countably compact
topological inverse semigroup~\cite{GutikRepovs2007}.

Moreover, the conditions were given in
\cite{BanakhDimitrovaGutik2009} and \cite{BanakhDimitrovaGutik20??}
when a countable compact or pseudocompact topological semigroup does
not contain the bicyclic semigroup. However, Banakh, Dimitrova and
Gutik constructed, using set-theoretic assumptions (Continuum
Hypothesis or Martin Axiom), an example of a Tychonoff countably
compact topological semigroup which contains the bicyclic
semigroup~\cite{BanakhDimitrovaGutik20??}.

We remark that the bicyclic semigroup is isomorphic to the
semigroup $\mathscr{C}_{\mathbb{N}}(\alpha,\beta)$ which is
generated by partial transformations $\alpha$ and $\beta$ of the
set of positive integers $\mathbb{N}$, defined as follows:
\begin{equation*}
    (n)\alpha=n+1  \quad \mbox{ if } \, n\geqslant 1, \qquad
    \mbox{and} \qquad
    (n)\beta=n-1   \quad \mbox{ if } \, n> 1.
\end{equation*}
Therefore the semigroup
$\mathscr{I}_{\infty}^{\!\nearrow}(\mathbb{N})$ contains an
isomorphic copy of the bicyclic semigroup ${\mathscr{C}}(p,q)$.

%%%%%%%%%%%%%%%%%%%%%%%%%%%%%%%%%%%%%%%%%%%%%%%%%%%%%%%%%%%%%%%%%%

\section{Algebraic properties of the semigroup
$\mathscr{I}_{\infty}^{\!\nearrow}(\mathbb{N})$}

\begin{proposition}\label{proposition-2.1} The following properties hold:
\begin{itemize}
    \item[$(i)$] $\mathscr{I}_{\infty}^{\!\nearrow}(\mathbb{N})$
         is a simple semigroup.

    \item[$(ii)$] $\alpha\mathscr{R}\beta$ $(\alpha\mathscr{L}\beta)$ in
         $\mathscr{I}_{\infty}^{\!\nearrow}(\mathbb{N})$ if and
         only if
         $\operatorname{dom}\alpha=\operatorname{dom}\beta$
         $(\operatorname{rank}\alpha=\operatorname{rank}\beta)$.

    \item[$(iii)$] $\alpha\mathscr{H}\beta$ in
          $\mathscr{I}_{\infty}^{\!\nearrow}(\mathbb{N})$ if and
          only if $\alpha=\beta$.

    \item[$(iv)$] For every $\varepsilon,\iota\in
          E(\mathscr{I}_{\infty}^{\!\nearrow}(\mathbb{N}))$ there
          exists $\alpha\in
          \mathscr{I}_{\infty}^{\!\nearrow}(\mathbb{N})$
          such that $\alpha\alpha^{-1}=\varepsilon$ and
          $\alpha^{-1}\alpha=\iota$.

    \item[$(v)$] $\mathscr{I}_{\infty}^{\!\nearrow}(\mathbb{N})$
    is a bisimple semigroup.

    \item[$(vi)$] If $\varepsilon,\iota\in
          E(\mathscr{I}_{\infty}^{\!\nearrow}(\mathbb{N}))$, then
          $\varepsilon\leqslant\iota$ if and only if
          $\operatorname{dom}\varepsilon\subseteq
          \operatorname{dom}\iota$.

    \item[$(vii)$] The semilattice
          $E(\mathscr{I}_{\infty}^{\!\nearrow}(\mathbb{N}))$ is
          isomorphic to $(\mathscr{P}_{<\omega}(\mathbb{N}),\subseteq)$ under
          the mapping $(\varepsilon)h=\mathbb{N}\setminus
          \operatorname{dom}\varepsilon$.

    \item[$(viii)$] Every maximal chain in
          $E(\mathscr{I}_{\infty}^{\!\nearrow}(\mathbb{N}))$ is an
          $\omega$-chain.
\end{itemize}
\end{proposition}

\begin{proof}
$(i)$ Let
 $
\alpha=
\left(%
\begin{array}{ccccc}
  n_1 & n_2 & n_3 & n_4 & \ldots \\
  m_1 & m_2 & m_3 & m_4 & \ldots \\
\end{array}%
\right)
 $
and
 $
\beta=
\left(%
\begin{array}{ccccc}
  k_1 & k_2 & k_3 & k_4 & \ldots \\
  l_1 & l_2 & l_3 & l_4 & \ldots \\
\end{array}%
\right)
 $
be any elements of the semigroup
$\mathscr{I}_{\infty}^{\!\nearrow}(\mathbb{N})$, where
$n_i,m_i,k_i,l_i\in\mathbb{N}$ for $i=1,2,3,\ldots$~.
We put
 $$
\gamma=
\left(%
\begin{array}{ccccc}
  k_1 & k_2 & k_3 & k_4 & \ldots \\
  n_1 & n_2 & n_3 & n_4 & \ldots \\
\end{array}%
\right) \ \
 \hbox{and} \ \
\delta=
\left(%
\begin{array}{ccccc}
  m_1 & m_2 & m_3 & m_4 & \ldots \\
  l_1 & l_2 & l_3 & l_4 & \ldots \\
\end{array}%
\right)
 .$$
Then we have that $\gamma\alpha\delta=\beta$. Therefore
$\mathscr{I}_{\infty}^{\!\nearrow}(\mathbb{N})\cdot\alpha\cdot
\mathscr{I}_{\infty}^{\!\nearrow}(\mathbb{N})=
\mathscr{I}_{\infty}^{\!\nearrow}(\mathbb{N})$ for any
$\alpha\in\mathscr{I}_{\infty}^{\!\nearrow}(\mathbb{N})$ and hence
$\mathscr{I}_{\infty}^{\!\nearrow}(\mathbb{N})$ is a simple
semigroup.

Statement $(ii)$ follows from  definitions of relations
$\mathscr{R}$ and $\mathscr{L}$ and Theorem~1.17 of
\cite{CP}. Also,
$(ii)$ implies $(iii)$.
For the idempotents $\varepsilon=
\left(%
\begin{array}{ccccc}
  m_1 & m_2 & m_3 & m_4 & \ldots \\
  m_1 & m_2 & m_3 & m_4 & \ldots \\
\end{array}%
\right)
 $
and $\iota=
\left(%
\begin{array}{ccccc}
  l_1 & l_2 & l_3 & l_4 & \ldots \\
  l_1 & l_2 & l_3 & l_4 & \ldots \\
\end{array}%
\right)
 $
we put
$\alpha=
\left(%
\begin{array}{ccccc}
  m_1 & m_2 & m_3 & m_4 & \ldots \\
  l_1 & l_2 & l_3 & l_4 & \ldots \\
\end{array}%
\right)
 $. Then $\alpha\alpha^{-1}=\varepsilon$
 and
$\alpha^{-1}\alpha=\iota$, and hence $(iv)$ holds. Also, $(v)$
follows from $(ii)$. All other assertions are trivial.
\end{proof}

\begin{proposition}\label{proposition-2.1a}
For every
$\alpha,\beta\in\mathscr{I}_{\infty}^{\!\nearrow}(\mathbb{N}),$
both sets
 $
\{\chi\in\mathscr{I}_{\infty}^{\!\nearrow}(\mathbb{N})\mid
\alpha\cdot\chi=\beta\}
 $
 and
 $
\{\chi\in\mathscr{I}_{\infty}^{\!\nearrow}(\mathbb{N})\mid
\chi\cdot\alpha=\beta\}
 $
are finite. Consequently, every right translation and every left
translation by an element of the semigroup
$\mathscr{I}_{\infty}^{\!\nearrow}(\mathbb{N})$ is a finite-to-one
map.
\end{proposition}

\begin{proof}
We denote
$A=\{\chi\in\mathscr{I}_{\infty}^{\!\nearrow}(\mathbb{N})\mid
\alpha\cdot\chi=\beta\}$ and
$B=\{\chi\in\mathscr{I}_{\infty}^{\!\nearrow}(\mathbb{N})\mid
\alpha^{-1}\cdot\alpha\cdot\chi=\alpha^{-1}\cdot\beta\}$. Then
$A\subseteq B$ and the restriction of any partial map $\chi\in B$
to $\operatorname{dom}(\alpha^{-1}\cdot\alpha)$ coincides with the
partial map $\alpha^{-1}\cdot\beta$. Since every partial map from
$\mathscr{I}_{\infty}^{\!\nearrow}(\mathbb{N})$ is monotone and
non-decreasing, the set $B$ is finite and hence so is $A$.
\end{proof}

For every $\gamma\in\mathscr{I}_{\infty}^{\!\nearrow}(\mathbb{N})$
 $
M_{\operatorname{dom}}(\gamma)=  \min\{n\in\mathbb{N}\mid
m\in\operatorname{dom}\gamma \mbox{ for all } m\geqslant n\}
 $
 and
 $
M_{\operatorname{ran}}(\gamma)=  \min\{n\in\mathbb{N}\mid
m\in\operatorname{ran}\gamma \mbox{ for all } m\geqslant n\}
 $
and put
 $
    M(\gamma)=\max\{M_{\operatorname{dom}}(\gamma),
    M_{\operatorname{ran}}(\gamma)\}.
 $

\begin{lemma}\label{lemma-2.2}
For every idempotent $\varepsilon$ of the semigroup
$\mathscr{I}_{\infty}^{\!\nearrow}(\mathbb{N})$ there exists an
idempotent $\varepsilon_0\in
E(\mathscr{I}_{\infty}^{\!\nearrow}(\mathbb{N})) \setminus
E(\mathscr{C}_{\mathbb{N}}(\alpha,\beta))$ such that the following
conditions hold:
\begin{enumerate}
    \item $\varepsilon_0\leqslant\varepsilon$;
    \item $\varepsilon_0$ is the unity of a subsemigroup
     $\mathscr{C}$ of
     $\mathscr{I}_{\infty}^{\!\nearrow}(\mathbb{N})$, which
     is isomorphic to the bicyclic semigroup;  \; and
    \item $\mathscr{C}\cap
     \mathscr{C}_{\mathbb{N}}(\alpha,\beta)=\varnothing$.
\end{enumerate}
\end{lemma}

\begin{proof}
Let $\varepsilon$ be an arbitrary idempotent of the semigroup
$\mathscr{I}_{\infty}^{\!\nearrow}(\mathbb{N})$. We put
$n_0=M(\varepsilon)+1$ and
$$%\begin{equation*}
 \varepsilon_0=
\left(%
\begin{array}{ccccc}
  n_0{-}1 & n_0{+}1 & n_0{+}2 & n_0{+}3 & \cdots \\
  n_0{-}1 & n_0{+}1 & n_0{+}2 & n_0{+}3 & \cdots \\
\end{array}%
\right). $$%\end{equation*}
We define the partial monotone bijections
$\widetilde{\alpha}\colon\mathbb{N}\rightharpoonup\mathbb{N}$ and
$\widetilde{\beta}\colon\mathbb{N}\rightharpoonup\mathbb{N}$ as
follows:
\begin{equation*}
    \widetilde{\alpha}=
\left(%
\begin{array}{ccccc}
  n_0{-}1 & n_0{+}1 & n_0{+}2 & n_0{+}3 & \cdots \\
  n_0{-}1 & n_0{+}2 & n_0{+}3 & n_0{+}4 & \cdots \\
\end{array}%
\right) \qquad \hbox{ and } \qquad \widetilde{\beta}=
\left(%
\begin{array}{ccccc}
  n_0{-}1 & n_0{+}2 & n_0{+}3 & n_0{+}4 & \cdots \\
  n_0{-}1 & n_0{+}1 & n_0{+}2 & n_0{+}3 & \cdots \\
\end{array}%
\right).
\end{equation*}
Let $\mathscr{C}$ a semigroup generated by
the elements $\widetilde{\alpha}$ and $\widetilde{\beta}$. Then
$\mathscr{C}$ satisfies the conditions (2)---(3) of the lemma and
$\varepsilon_0=\widetilde{\alpha}\cdot\widetilde{\beta}$ is the
identity of the semigroup $\mathscr{C}$ such that
$\varepsilon_0\leqslant\varepsilon$.
\end{proof}

\begin{lemma}\label{lemma-2.4}
For every
$\lambda\in\mathscr{I}_{\infty}^{\!\nearrow}(\mathbb{N})$ there
exist $\mu\in\mathscr{C}_{\mathbb{N}}(\alpha,\beta)$ and
$\varepsilon\in E(\mathscr{C}_{\mathbb{N}}(\alpha,\beta))$ such
that $\lambda\cdot\varepsilon=\mu\cdot\varepsilon$ and
$\varepsilon\cdot\lambda=\varepsilon\cdot\mu$.
\end{lemma}

\begin{proof}
The definition of the semigroup
$\mathscr{I}_{\infty}^{\!\nearrow}(\mathbb{N})$ implies that for
every $\gamma\in\mathscr{I}_{\infty}^{\!\nearrow}(\mathbb{N})$ the
notions $M_{\operatorname{dom}}(\gamma)$,
$M_{\operatorname{ran}}(\gamma)$ and $M(\gamma)$ exist and they
are unique, and hence they are well-defined.

We define partial maps
$\mu\colon\mathbb{N}\rightharpoonup\mathbb{N}$ and
$\varepsilon\colon\mathbb{N}\rightharpoonup\mathbb{N}$ as follows:
$\operatorname{dom}\mu=\{ n\in\mathbb{N}\mid n\geqslant
M(\lambda)\}$ and $(i)\mu=(i)\lambda$ for all
$i\in\operatorname{dom}\mu$ and $\operatorname{dom}\varepsilon=\{
n\in\mathbb{N}\mid n\geqslant M(\lambda)\}$ and $(i)\mu=i$ for all
$i\in\operatorname{dom}\mu$. Then we have
$\lambda\cdot\varepsilon=\mu\cdot\varepsilon$ and
$\varepsilon\cdot\lambda=\varepsilon\cdot\mu$.
\end{proof}

The proof of the following lemma is similar to the proof of
Lemma~\ref{lemma-2.4}.

\begin{lemma}\label{lemma-2.4a}
For every idempotent $\varphi\in
E(\mathscr{I}_{\infty}^{\!\nearrow}(\mathbb{N}))$ there exists
$\varepsilon\in E(\mathscr{C}_{\mathbb{N}}(\alpha,\beta))$ such
that $\varphi\cdot\varepsilon\in
E(\mathscr{C}_{\mathbb{N}}(\alpha,\beta))$. More than,
$\psi\cdot\varepsilon\in
E(\mathscr{C}_{\mathbb{N}}(\alpha,\beta))$ for every $\psi\in
E(\mathscr{C}_{\mathbb{N}}(\alpha,\beta))$ such that
$\psi\leqslant\varepsilon$.
\end{lemma}

\begin{lemma}\label{lemma-2.4b}
For every idempotent $\varepsilon\in
E(\mathscr{I}_{\infty}^{\!\nearrow}(\mathbb{N}))$ there exists
$\varphi\in E(\mathscr{C}_{\mathbb{N}}(\alpha,\beta))$ such that
$\varphi\leqslant\varepsilon$.
\end{lemma}

\begin{proof}
The definition of $\mathscr{I}_{\infty}^{\!\nearrow}(\mathbb{N})$
implies that there exists a maximal positive integer
$n_{\varepsilon}$ such that
$n_{\varepsilon}-1\notin\operatorname{dom}\varepsilon$. We put
 $
    \varphi=\left(%
\begin{array}{ccccc}
  n_\varepsilon & n_\varepsilon{+}1 & n_\varepsilon{+}2 & n_\varepsilon{+}3 & \cdots \\
  n_\varepsilon & n_\varepsilon{+}1 & n_\varepsilon{+}2 & n_\varepsilon{+}3 & \cdots \\
\end{array}%
\right)
 $
Then we get $\varphi\in E(\mathscr{C}_{\mathbb{N}}(\alpha,\beta))$
and $\varphi\leqslant\varepsilon$.
\end{proof}

\begin{lemma}\label{lemma-2.4c}
For every element $\lambda\in
\mathscr{I}_{\infty}^{\!\nearrow}(\mathbb{N})$ there exists an
idempotent $\varepsilon$ of the subsemigroup
$\mathscr{C}_{\mathbb{N}}(\alpha,\beta)$ such that
$\lambda\cdot\varepsilon\cdot\lambda^{-1},
\lambda^{-1}\cdot\varepsilon\cdot\lambda\in
E(\mathscr{C}_{\mathbb{N}}(\alpha,\beta))$.
\end{lemma}

\begin{proof}
By Lemma~\ref{lemma-2.4} there exists
$\mu\in\mathscr{C}_{\mathbb{N}}(\alpha,\beta)$ and $\varepsilon\in
E(\mathscr{C}_{\mathbb{N}}(\alpha,\beta))$ such that
$\lambda\cdot\varepsilon=\mu\cdot\varepsilon$ and
$\varepsilon\cdot\lambda=\varepsilon\cdot\mu$. Therefore, we have
that
\begin{equation*}
    \lambda\cdot\varepsilon\cdot\lambda^{-1} = \,
    \lambda\cdot\varepsilon\cdot\varepsilon\cdot\lambda^{-1}=
    (\lambda\cdot\varepsilon)\cdot(\lambda\cdot\varepsilon)^{-1}=
    (\mu\cdot\varepsilon)\cdot(\mu\cdot\varepsilon)^{-1}=
    \mu\cdot\varepsilon\cdot\varepsilon\cdot\mu^{-1}=
    \mu\cdot\varepsilon\cdot\mu^{-1}\in
    E(\mathscr{C}_{\mathbb{N}}(\alpha,\beta))
\end{equation*}
and similarly
$
    \lambda^{-1}\cdot\varepsilon\cdot\lambda=\mu^{-1}\cdot\varepsilon\cdot\mu\in
    E(\mathscr{C}_{\mathbb{N}}(\alpha,\beta)).
$
\end{proof}

\begin{lemma}\label{lemma-2.3}
Let $S$ be a semigroup and
$h\colon\mathscr{I}_{\infty}^{\!\nearrow}(\mathbb{N})\rightarrow
S$  a homomorphism such that $(\varepsilon)h=(\varphi)h$ for some
distinct idempotents $\varepsilon,\varphi\in
E(\mathscr{I}_{\infty}^{\!\nearrow}(\mathbb{N}))$. Then
$(\varepsilon)h=(\psi)h,$ for every $\psi\in
E(\mathscr{I}_{\infty}^{\!\nearrow}(\mathbb{N}))$.
\end{lemma}

\begin{proof}
We consider the following cases:
\begin{itemize}
    \item[(1)] $\varepsilon,\varphi\in
      E(\mathscr{C}_{\mathbb{N}}(\alpha,\beta))$;
    \item[(2)] $\varepsilon$ and $\varphi$ are distinct comparable
      idempotents of
      $E(\mathscr{I}_{\infty}^{\!\nearrow}(\mathbb{N}))$; \;
    \item[(3)] $\varepsilon$ and $\varphi$ are distinct
      incomparable idempotents of
      $E(\mathscr{I}_{\infty}^{\!\nearrow}(\mathbb{N}))$.
\end{itemize}

Suppose case (1) holds. Then by Corollary~1.32 of \cite{CP} we have
that $(\chi)h=(\varepsilon)h$ for every $\chi\in
E(\mathscr{C}_{\mathbb{N}}(\alpha,\beta))$. Let $\psi$ be any
idempotent of $E(\mathscr{I}_{\infty}^{\!\nearrow}(\mathbb{N}))$
such that $\psi\notin E(\mathscr{C}_{\mathbb{N}}(\alpha,\beta))$.
Then by Proposition~\ref{proposition-2.1}$(v)$ there exists
$\gamma\in\mathscr{I}_{\infty}^{\!\nearrow}(\mathbb{N})$ such that
$\gamma\cdot\gamma^{-1}=\varepsilon$ and
$\gamma^{-1}\cdot\gamma=\psi$. By Lemma~\ref{lemma-2.4c} there
exist $\varepsilon_0,\varepsilon_1\in
E(\mathscr{C}_{\mathbb{N}}(\alpha,\beta))$ such that
$\gamma^{-1}\cdot\varepsilon_0\cdot\gamma=\varepsilon_1\in
E(\mathscr{C}_{\mathbb{N}}(\alpha,\beta))$. Therefore, we have
that
\begin{equation*}
\begin{split}
 (\psi)h &\,=(\psi\cdot\psi)h=
 (\gamma^{-1}\cdot\gamma\cdot\gamma^{-1}\cdot\gamma)h=
 (\gamma^{-1}\cdot\varepsilon\cdot\gamma)h=
 (\gamma^{-1})h\cdot(\varepsilon)h\cdot(\gamma)h=\\
         &\,=
             (\gamma^{-1})h\cdot(\varepsilon_0)h\cdot(\gamma)h=
             (\gamma^{-1}\cdot\varepsilon_0\cdot\gamma)h=
             (\varepsilon_1)h=(\varepsilon)h.
\end{split}
\end{equation*}

Suppose case (2) holds. Without loss of generality we can assume
that $\varepsilon\leqslant\varphi$. Then
$(\varepsilon)h=(\varepsilon_1)h=(\varphi)h$ for every idempotent
$\varepsilon_1\in
E(\mathscr{I}_{\infty}^{\!\nearrow}(\mathbb{N}))$ such that
$\varepsilon\leqslant\varepsilon_1\leqslant\varphi$. Therefore,
without loss of generality we can assume that
$\operatorname{dom}\varphi\setminus\operatorname{dom}\varepsilon$
is singleton. Let $\{n^\varphi_\varepsilon\}=
\operatorname{dom}\varphi\setminus\operatorname{dom}\varepsilon$.
Let $j$ be the minimal integer of $\operatorname{dom}\varepsilon$
such that $(i)\varepsilon=i$ for all $i\geqslant j$. We put
$%\begin{equation*}
    \varepsilon_0=
    \left(%
\begin{array}{ccccc}
  n^\varphi_\varepsilon & j & j{+}1 & j{+}2 & \cdots \\
  n^\varphi_\varepsilon & j & j{+}1 & j{+}2 & \cdots \\
\end{array}%
\right) \mbox{ and } \lambda=
\left(%
\begin{array}{ccccc}
  n^\varphi_\varepsilon & j & j{+}1 & j{+}2 & \cdots \\
  j{-}1                 & j & j{+}1 & j{+}2 & \cdots \\
\end{array}%
\right). $ %\end{equation*}
  Then
$\lambda^{-1}\cdot\varepsilon_0\cdot\varphi\cdot
\varepsilon_0\cdot\lambda$ and
$\lambda^{-1}\cdot\varepsilon_0\cdot\varepsilon\cdot
\varepsilon_0\cdot\lambda$ are distinct idempotents of the
subsemigroup $\mathscr{C}_{\mathbb{N}}(\alpha,\beta)$. Therefore,
we have that
\begin{equation*}
    (\lambda^{-1}\cdot\varepsilon_0\cdot\varphi\cdot
    \varepsilon_0\cdot\lambda)h =
        (\lambda^{-1}\cdot\varepsilon_0)h\cdot(\varphi)h\cdot
    (\varepsilon_0\cdot\lambda)h=
        (\lambda^{-1}\cdot\varepsilon_0)h\cdot(\varepsilon)h\cdot
    (\varepsilon_0\cdot\lambda)h=
        (\lambda^{-1}\cdot\varepsilon_0\cdot\varepsilon\cdot
    \varepsilon_0\cdot\lambda)h,
\end{equation*}
and hence case (1) holds.

Suppose case (3) holds. Then we have that
 $
 (\varepsilon)h=(\varepsilon\cdot\varepsilon)h=
 (\varepsilon)h\cdot(\varepsilon)h=
 (\varepsilon)h\cdot(\varphi)h= (\varepsilon\cdot\varphi)h.
 $
Since the idempotents $\varepsilon$ and $\varphi$ are distinct and
incomparable we conclude that $\varepsilon\cdot\varphi<\varepsilon$
and $\varepsilon\cdot\varphi<\varphi$, and hence case (2) holds.
\end{proof}

\begin{theorem}\label{theorem-2.5}
Let $S$ be a semigroup and
$h\colon\mathscr{I}_{\infty}^{\!\nearrow}(\mathbb{N})\rightarrow
S$  a non-annihilating homomorphism. Then either $h$ is a
monomorphism or $(\mathscr{I}_{\infty}^{\!\nearrow}(\mathbb{N}))h$
is a cyclic subgroup of $S$.
\end{theorem}

\begin{proof}
Suppose that
$h\colon\mathscr{I}_{\infty}^{\!\nearrow}(\mathbb{N})\rightarrow
S$ is not an isomorphism ``into''. Then $(\alpha)h=(\beta)h$, for
some distinct
$\alpha,\beta\in\mathscr{I}_{\infty}^{\!\nearrow}(\mathbb{N})$.
Since $\mathscr{I}_{\infty}^{\!\nearrow}(\mathbb{N})$ is an
inverse semigroup we conclude that
\begin{equation*}
    \left(\alpha^{-1}\right)h=\big((\alpha)h\big)^{-1}=
    \big((\beta)h\big)^{-1}=\left(\beta^{-1}\right)h
\end{equation*}
and hence $(\alpha\alpha^{-1})h=(\beta\beta^{-1})h$. Therefore the
assertion of Lemma~\ref{lemma-2.3} holds. Since every homomorphic
image of an inverse semigroup is an inverse semigroup we conclude
that $(\mathscr{I}_{\infty}^{\!\nearrow}(\mathbb{N}))h$ is a
subgroup of $S$.

Since the map
$h\colon\mathscr{I}_{\infty}^{\!\nearrow}(\mathbb{N})\rightarrow
S$ is a group homomorphism we have that $h$ generates a group
congruence $\frak{h}$ on
$\mathscr{I}_{\infty}^{\!\nearrow}(\mathbb{N})$. If $\mathfrak{c}$
is any congruence on the semigroup
$\mathscr{I}_{\infty}^{\!\nearrow}(\mathbb{N})$ then the mapping
$\mathfrak{c}\mapsto\mathfrak{c}\vee\mathfrak{g}$ maps the
congruence $\mathfrak{c}$ onto a group congruence
$\mathfrak{c}\vee\mathfrak{g}$, where $\mathfrak{g}$ is the least
group congruence on
$\mathscr{I}_{\infty}^{\!\nearrow}(\mathbb{N})$ (cf.
\cite[Section~III]{Petrich1984}).

Such a mapping is a map from the lattice of all congruences of the
semigroup $\mathscr{I}_{\infty}^{\!\nearrow}(\mathbb{N})$ onto the
lattice of all group congruences of
$\mathscr{I}_{\infty}^{\!\nearrow}(\mathbb{N})$
~\cite{Petrich1984}. By Lemma~III.5.2 of~\cite{Petrich1984}, the
elements $\gamma$ and $\delta$ of the semigroup
$\mathscr{I}_{\infty}^{\!\nearrow}(\mathbb{N})$ are
$\mathfrak{g}$-equivalent if and only if there exists an
idempotent $\varepsilon\in
E(\mathscr{I}_{\infty}^{\!\nearrow}(\mathbb{N}))$ such that
$\gamma\cdot\varepsilon=\delta\cdot\varepsilon$.
Lemma~\ref{lemma-2.4} implies that for every
$\gamma\in\mathscr{I}_{\infty}^{\!\nearrow}(\mathbb{N})$ there
exists $\delta\in\mathscr{C}_{\mathbb{N}}(\alpha,\beta)$ such that
$\gamma\mathfrak{g}\delta$. Therefore the least group congruence
$\mathfrak{g}$ on $\mathscr{I}_{\infty}^{\!\nearrow}(\mathbb{N})$
induces the least group congruence on its subsemigroup
$\mathscr{C}_{\mathbb{N}}(\alpha,\beta)$.

We observe that $\gamma\mathfrak{g}\delta$ in
$\mathscr{I}_{\infty}^{\!\nearrow}(\mathbb{N})$ (or in
$\mathscr{C}_{\mathbb{N}}(\alpha,\beta)$) if and only if there
exists a positive integer $i$ such that the restrictions of the
partial mapping $\gamma$ and $\delta$ onto the set
$\{i,i+1,i+2,\ldots\}$ coincide. Then we define the map
$f\colon\mathscr{I}_{\infty}^{\!\nearrow}(\mathbb{N})\rightarrow
\mathbb{Z}_+$ onto the additive group of integers as follow:
\begin{equation}\label{equation-2.a}
    (\gamma)f=n \quad \mbox{ if } \; (i)\gamma=i+n \;
    \mbox{ for infinitely many positive integers } \; i.
\end{equation}
The definition of the semigroup
$\mathscr{I}_{\infty}^{\!\nearrow}(\mathbb{N})$ implies that such
a map $f$ is well-defined. The map
$f\colon\mathscr{I}_{\infty}^{\!\nearrow}(\mathbb{N})\rightarrow
\mathbb{Z}_+$ generates the least group congruence $\mathfrak{g}$
on the semigroup $\mathscr{I}_{\infty}^{\!\nearrow}(\mathbb{N})$
and hence $f$ is a group homomorphism. This completes the proof of
the theorem.
\end{proof}

\begin{remark}\label{remerk-2.6}
We observe that the following conditions are equivalent:
\begin{itemize}
    \item[$(i)$] $\gamma\mathfrak{g}\delta$ in
$\mathscr{I}_{\infty}^{\!\nearrow}(\mathbb{N})$;
    \item[$(ii)$] there exists
$\varepsilon\in E(\mathscr{I}_{\infty}^{\!\nearrow}(\mathbb{N}))$
such that $\varepsilon\cdot\gamma=\varepsilon\cdot\delta$;
    \item[$(iii)$] there exists
$\varepsilon\in E(\mathscr{C}_{\mathbb{N}}(\alpha,\beta))$ such
that $\varepsilon\cdot\gamma=\varepsilon\cdot\delta$; and
    \item[$(iv)$] there exists
$\varepsilon\in E(\mathscr{C}_{\mathbb{N}}(\alpha,\beta))$ such
that $\gamma\cdot\varepsilon=\delta\cdot\varepsilon$.
\end{itemize}
\end{remark}

%%%%%%%%%%%%%%%%%%%%%%%%%%%%%%%%%%%%%%%%%%%%%%%%%%%%%%%%%%%%%%%

\section{On  topological semigroup
$\mathscr{I}_{\infty}^{\!\nearrow}(\mathbb{N})$}

\begin{lemma}\label{referee}
If $E$ is an infinite semilattice satisfying that
${\uparrow}e$ is
finite for all $e\in E$, then the only locally compact, Hausdorff
topology relative to which $E$ is a topological semilattice is the
discrete topology.
\end{lemma}

\begin{proof}
Assume that $E$ has a locally compact, Hausdorff topology under
which it is a topological semilattice, and that $E$ has a
non-isolated point $e$ in this topology. Since $E$ is Hausdorff,
every neighbourhood of $e$ has infinitely many elements. Let $K$
be a compact neighbourhood of $e$. Then there is a net $\langle
e_i\rangle_{i\in{\mathscr{I}}}\subseteq\operatorname{Int}_E(K)\setminus\{
e\}$ with $\lim_ie_i=e$. Since $E$ is a topological semilattice,
$e=e^2=e\cdot\lim_ie_i=\lim_i(e\cdot e_i)$, and since
${\uparrow}e$ is finite, we can assume that $e\cdot e_i<e$ are
distinct for all $i\in\mathscr{I}$. Since
$e\in\operatorname{Int}_E(K)$, we can also assume $e\cdot
e_i\in\operatorname{Int}_E(K)$ for all $i$. Thus, we can assume
$e_i<e$ satisfy $e_i\in\operatorname{Int}_E(K)$ for all $i$.

Choose one such $e_j$, and then note that $e_j=e_j\cdot
e=e_j\cdot\lim_ie_i=\lim_i(e_j\cdot e_i)$. The same argument we
just gave for $e$ then implies that $e_i<e_j$ for all $i$, that
$\lim_ie_i=e_j$, and that $e_i\in\operatorname{Int}_E(K)$ for all
$i$. We let $e_1=e_j$, and now repeat the argument. Since $K$ is
compact, this sequence has a limit point, $s$ in $K$, and the
continuity of the semilattice operation implies $s$ is another
idempotent and that $s< e_n$ for all $n$. But then ${\uparrow}s$
is infinite, contrary to our hypothesis. Hence $E$ cannot have a
nonisolated point, so it is discrete.
\end{proof}

Proposition~\ref{proposition-2.1a} and Lemma~\ref{referee} imply
the following:

\begin{lemma}\label{lemma-3.2}
If $\mathscr{I}_{\infty}^{\!\nearrow}(\mathbb{N})$ is a locally
compact Hausdorff topological semigroup, then
$E(\mathscr{I}_{\infty}^{\!\nearrow}(\mathbb{N}))$ with the
induces from $\mathscr{I}_{\infty}^{\!\nearrow}(\mathbb{N})$
topology is a discrete topological semilattice.
\end{lemma}

\begin{theorem}\label{theorem-3.3}
Every locally compact Hausdorff topology on the semigroup
$\mathscr{I}_{\infty}^{\!\nearrow}(\mathbb{N})$ such that
$\mathscr{I}_{\infty}^{\!\nearrow}(\mathbb{N})$ is a topological
inverse semigroup,
is discrete.
\end{theorem}

\begin{proof}
By Lemma~\ref{lemma-3.2},
the band
$E(\mathscr{I}_{\infty}^{\!\nearrow}(\mathbb{N}))$ is a discrete
topological space. Since
$\mathscr{I}_{\infty}^{\!\nearrow}(\mathbb{N})$ is a topological
inverse semigroup, the maps
$h\colon\mathscr{I}_{\infty}^{\!\nearrow}(\mathbb{N})\rightarrow
E(\mathscr{I}_{\infty}^{\!\nearrow}(\mathbb{N}))$ and
$f\colon\mathscr{I}_{\infty}^{\!\nearrow}(\mathbb{N})\rightarrow
E(\mathscr{I}_{\infty}^{\!\nearrow}(\mathbb{N}))$ defines by the
formulae $(\alpha)h=\alpha\cdot\alpha^{-1}$ and
$(\alpha)f=\alpha^{-1}\cdot\alpha$ are continuous and for every
two idempotents $\varepsilon$ and $\varphi$ of the semigroup
$\mathscr{I}_{\infty}^{\!\nearrow}(\mathbb{N})$ there exists a
unique element $\chi$ in
$\mathscr{I}_{\infty}^{\!\nearrow}(\mathbb{N})$ such that
$\chi\cdot\chi^{-1}=\varepsilon$ and $\chi^{-1}\cdot\chi=\varphi$,
we have that every element of the topological semigroup
$\mathscr{I}_{\infty}^{\!\nearrow}(\mathbb{N})$ is an isolated
point of the topological space
$\mathscr{I}_{\infty}^{\!\nearrow}(\mathbb{N})$.
\end{proof}

The following theorem describes the closure of the discrete
semigroup $\mathscr{I}_{\infty}^{\!\nearrow}(\mathbb{N})$ in a
topological semigroup.

\begin{theorem}\label{theorem-3.4}
If a topological semigroup S contains
$\mathscr{I}_{\infty}^{\!\nearrow}(\mathbb{N})$ as a proper, dense
discrete subsemigroup, then
$\mathscr{I}_{\infty}^{\!\nearrow}(\mathbb{N})$ is an open
subsemigroup of $S$ and
$S\setminus\mathscr{I}_{\infty}^{\!\nearrow}(\mathbb{N})$ is an
ideal of $S$.
\end{theorem}

\begin{proof}
The first assertion of the theorem follows from Theorem~3.3.9 of
\cite{Engelking1989}.

Suppose that $\chi\in
S\setminus\mathscr{I}_{\infty}^{\!\nearrow}(\mathbb{N})$ and
$\alpha\in S$. If
$\chi\cdot\alpha\in\mathscr{I}_{\infty}^{\!\nearrow}(\mathbb{N})$
then there exist open neighbourhoods $U(\chi)$ and $U(\alpha)$ of
$\chi$ and $\alpha$ in $S$, respectively, such that $U(\chi)\cdot
U(\alpha)=\{\chi\cdot\alpha\}$. We observe that the set
$U(\chi)\cap\mathscr{I}_{\infty}^{\!\nearrow}(\mathbb{N})$ is
infinite and fix any point $\mu\in
U(\alpha)\cap\mathscr{I}_{\infty}^{\!\nearrow}(\mathbb{N})$. Hence
we have
\begin{equation*}
   (U(\chi)\cap\mathscr{I}_{\infty}^{\!\nearrow}(\mathbb{N}))
   \cdot\mu\subseteq
   (U(\chi)\cap\mathscr{I}_{\infty}^{\!\nearrow}(\mathbb{N}))
   \cdot
   (U(\alpha)\cap\mathscr{I}_{\infty}^{\!\nearrow}(\mathbb{N}))=
   \{\chi\cdot\alpha\}.
\end{equation*}
This contradicts Proposition~\ref{proposition-2.1a}. The obtained
contradiction implies that $\chi\cdot\alpha\in
S\setminus\mathscr{I}_{\infty}^{\!\nearrow}(\mathbb{N})$.

The proof of the assertion $\alpha\cdot\chi\in
S\setminus\mathscr{I}_{\infty}^{\!\nearrow}(\mathbb{N})$ is
similar.
\end{proof}

Theorems~\ref{theorem-3.3} and \ref{theorem-3.4} imply the
following:

\begin{corollary}\label{corollary-3.5}
If a topological semigroup S contains
$\mathscr{I}_{\infty}^{\!\nearrow}(\mathbb{N})$ as a proper, dense
locally compact subsemigroup,  then
$\mathscr{I}_{\infty}^{\!\nearrow}(\mathbb{N})$ is an open
subsemigroup of $S$ and
$S\setminus\mathscr{I}_{\infty}^{\!\nearrow}(\mathbb{N})$ is an
ideal of $S$.
\end{corollary}

The following example shows that a remainder of a closure of the
discrete (and hence of a locally compact topological inverse)
semigroup $\mathscr{I}_{\infty}^{\!\nearrow}(\mathbb{N})$ in a
topological inverse semigroup contains only a zero element.

\begin{example}\label{example-3.6}
Let $\mathscr{I}_{\infty}^{\!\nearrow}(\mathbb{N})$ be a discrete
topological semigroup and let $S$ be the semigroup
$\mathscr{I}_{\infty}^{\!\nearrow}(\mathbb{N})$ with adjoined zero
$\mathscr{O}$, i.e.
$\mathscr{O}\cdot\alpha=\alpha\cdot\mathscr{O}=
\mathscr{O}\cdot\mathscr{O}=\mathscr{O}$ for all $\alpha\in
\mathscr{I}_{\infty}^{\!\nearrow}(\mathbb{N})$. We define a
topology $\tau$ on $S$ as follows:
\begin{itemize}
    \item[(a)] all non-zero elements of the semigroup $S$ are
    isolated points in $(S,\tau)$; and
    \item[(b)] the family
    $\mathscr{B}(\mathscr{O})=\{U_i(\mathscr{O})\mid
    i=1,2,3,\ldots \}$, where
    $
      U_i(\mathscr{O})=\{\mathscr{O}\}\cup\{\alpha\in
      \mathscr{I}_{\infty}^{\!\nearrow}(\mathbb{N})\mid
      |\mathbb{N}\setminus\operatorname{dom}\alpha|\geqslant i
      \; \mbox{ and } \;
      |\mathbb{N}\setminus\operatorname{ran}\alpha|\geqslant
      i\},
    $
    determines a base of the topology $\tau$ at the point
    $\mathscr{O}\in S$.
\end{itemize}

We observe
that
$(S,\tau)$ is a topological inverse semigroup which
contains $\mathscr{I}_{\infty}^{\!\nearrow}(\mathbb{N})$ as a
dense subsemigroup and $\tau$ is not a locally compact topology on
$\mathscr{I}_{\infty}^{\!\nearrow}(\mathbb{N})$.
\end{example}

We recall that a topological space X is said to be:
\begin{itemize}
    \item \emph{countably compact} if each closed discrete
          subspace of $X$ is finite;
    \item \emph{pseudocompact} if $X$ is Tychonoff and each
          continuous real-valued function on $X$ is bounded;
    \item \emph{sequentially compact} if each sequence
          $\{x_n\}_{n\in\mathbb{N}}\subseteq X$ has a
          convergent subsequence.
\end{itemize}

A topological semigroup $S$ is called
$\Gamma$-compact if for every $x\in S$ the closure of the set
$\{x,x^2,x^3,\ldots\}$ is a compactum in $S$ (see
\cite{HildebrantKoch1988}). Since the semigroup
$\mathscr{I}_{\infty}^{\!\nearrow}(\mathbb{N})$ contains the
bicyclic semigroup as a subsemigroup the results obtained in
\cite{AHK}, \cite{BanakhDimitrovaGutik2009},
\cite{BanakhDimitrovaGutik20??}, \cite{GutikRepovs2007},
\cite{HildebrantKoch1988} imply that \emph{if a topological
semigroup $S$ satisfies one of the following conditions: $(i)$~$S$
is compact; $(ii)$~$S$ is $\Gamma$-compact; $(iii)$~the square
$S\times S$ is countably compact; $(iv)$~$S$ is a countably
compact topological inverse semigroup; or $(v)$~the square
$S\times S$ is a Tychonoff pseudocompact space, then $S$ does not
contain the semigroup
$\mathscr{I}_{\infty}^{\!\nearrow}(\mathbb{N})$.}

The following example shows that the semigroup
$\mathscr{I}_{\infty}^{\!\nearrow}(\mathbb{N})$ embeds into a
locally compact topological semigroup as a discrete subsemigroup.

\begin{example}\label{example-3.8}
Let $\mathbb{Z}_+$ be the additive group of integers. Let
$h\colon\mathscr{I}_{\infty}^{\!\nearrow}(\mathbb{N})\rightarrow
\mathbb{Z}_+$ be a group homomorphism defined by the formula
(\ref{equation-2.a}). On the set
$S=\mathscr{I}_{\infty}^{\!\nearrow}(\mathbb{N})
\sqcup\mathbb{Z}_+$ we define the semigroup operation `$\ast$' as
follows:
\begin{equation*}
    x\ast y=
\left\{%
\begin{array}{cl}
    x\cdot y, & \hbox{ if }\;
     x,y\in\mathscr{I}_{\infty}^{\!\nearrow}(\mathbb{N});,\\
    x+(y)h, & \hbox{ if }\; x\in\mathbb{Z}_+ \;\mbox{ and }\;
     y\in\mathscr{I}_{\infty}^{\!\nearrow}(\mathbb{N});\\
    (x)h+y, & \hbox{ if }\;
     y\in\mathscr{I}_{\infty}^{\!\nearrow}(\mathbb{N}) \;\mbox{ and
    }\;  x\in\mathbb{Z}_+ ;\\
    (x)h+(y)h, & \hbox{ if }\; x,y\in\mathbb{Z}_+,\\
\end{array}%
\right.
\end{equation*}
where `$\cdot$' and `$+$' are the semigroup operation in
$\mathscr{I}_{\infty}^{\!\nearrow}(\mathbb{N})$ and the group
operation in $\mathbb{Z}_+$, respectively. The semigroup $S$ is
called the \emph{adjunction semigroup of}
$\mathscr{I}_{\infty}^{\!\nearrow}(\mathbb{N})$ \emph{and}
$\mathbb{Z}_+$ \emph{relative to homomorphism} $h$ (see
\cite[Vol.~1, pp.~77--80]{CHK}).

Let $\leqslant_c$ be the canonical partial order on the semigroup
$\mathscr{I}_{\infty}^{\!\nearrow}(\mathbb{N})$ (see
\cite[Section~II.1]{Petrich1984}), i.~e.
\begin{equation*}
    \alpha\leqslant_c\beta \qquad \mbox{ if and only if there
    exists }\quad \varepsilon\in
    E(\mathscr{I}_{\infty}^{\!\nearrow}(\mathbb{N})) \quad
    \mbox{ such that } \quad \alpha=\beta\cdot\varepsilon.
\end{equation*}
We observe that if $\alpha\leqslant_c\beta$ in
$\mathscr{I}_{\infty}^{\!\nearrow}(\mathbb{N})$, then
$(\alpha)h=(\beta)h$. For every $x\in\mathbb{Z}_+$ and $\alpha\in
\mathscr{I}_{\infty}^{\!\nearrow}(\mathbb{N})$ such that
$(\alpha)h=x$ we put
\begin{equation*}
    U_{\alpha}(x)=\{x\}\cup\{\beta\in
    \mathscr{I}_{\infty}^{\!\nearrow}(\mathbb{N})\mid(\beta)h=x
    \mbox{ and } \alpha\nleqslant_c\beta\}.
\end{equation*}

We define a topology $\tau$ on $S$ as follows:
\begin{itemize}
    \item[$(i)$] all elements of the subsemigroup
    $\mathscr{I}_{\infty}^{\!\nearrow}(\mathbb{N})$
    are isolated points in $(S,\tau)$; and
    \item[$(ii)$] the family
    $\mathscr{B}(x)=\{U_{\alpha}(x)\mid(\alpha)h=x\}$  determines
    a base of the topology at the point $x\in\mathbb{Z}_+$.
\end{itemize}

Simple verifications show that $(S,\tau)$ is a $0$-dimensional
locally compact topological inverse semigroup.
\end{example}
%%%%%%%%%%%%%%%%%%%%%%%%%%%%%%%%%%%%%%%%%%%%%%%%%%%%%%%%%%%

\section*{Acknowledgements}

This research was support by the Slovenian Research Agency grants
P1-0292-0101,  J1-2057-0101 and BI-UA/09-10/005. We thank the
referee for many comments and suggestions. In particular, we are very
grateful to the referee for the new statement and proof of Lemma~\ref{referee}.

%%%%%%%%%%%%%%%%%%%%%%%%%%%%%%%%%%%%%%%%%%%%%%%%%%%%%%%%%%%%%%%


\begin{thebibliography}{19}

\bibitem{Andersen} O.~Andersen, {\em Ein Bericht \"{u}ber die
Struktur abstrakter Halbgruppen}, PhD Thesis, Hamburg, 1952.

\bibitem{AHK} L.~W.~Anderson, R.~P.~Hunter and R.~J.~Koch,
\textit{Some results on stability in semigroups}. Trans. Amer.
Math. Soc. {\bf 117} (1965), 521--529.

\bibitem{BanakhDimitrovaGutik2009} T.~Banakh, S.~Dimitrova and
O.~Gutik, \emph{The Rees-Suschkiewitsch Theorem for simple
topological semigroups}, Mat. Stud. \textbf{31}:2 (2009),
211--218.

\bibitem{BanakhDimitrovaGutik20??} T.~Banakh, S.~Dimitrova and
O.~Gutik, \emph{Embedding the bicyclic semigroup into countably
compact topological semigroups}, Topology Appl. (submitted)
(arXiv:0811.4276).

\bibitem{CHK} J.~H.~Carruth, J.~A.~Hildebrant and  R.~J.~Koch,
\emph{The Theory of Topological Semigroups}, Vol. I, Marcel
Dekker, Inc., New York and Basel, 1983; Vol. II, Marcel Dekker,
Inc., New York and Basel, 1986.


\bibitem{CP} A.~H.~Clifford and  G.~B.~Preston, \emph{The
Algebraic Theory of Semigroups}, Vol. I., Amer. Math. Soc. Surveys
7, Providence, R.I., 1961; Vol. II., Amer. Math. Soc. Surveys 7,
Providence, R.I., 1967.

\bibitem{EberhartSelden1969} C.~Eberhart and J.~Selden, {\em On
the closure of the bicyclic semigroup}, Trans. Amer. Math. Soc.
{\bf 144} (1969), 115--126.


\bibitem{Engelking1989} R.~Engelking, \emph{General Topology},
2nd ed., Heldermann, Berlin, 1989.

\bibitem{GutikRepovs2007} O.~Gutik and D.~Repov\v{s}, \emph{On
countably compact $0$-simple topological inverse semigroups},
Semigroup Forum \textbf{75}:2 (2007), 464--469.

\bibitem{HildebrantKoch1988} J.~A.~Hildebrant and R.~J.~Koch,
{\it Swelling actions of $\Gamma$-compact semigroups}, Semigroup
Forum {\bf 33} (1988), 65--85.

\bibitem{Petrich1984} M.~Petrich, \emph{Inverse Semigroups}, John
Wiley $\&$ Sons, New York, 1984.

\bibitem{Wagner1952} V.~V.~Wagner, \emph{Generalized groups},
Dokl. Akad. Nauk SSSR \textbf{84} (1952), 1119--1122 (in
Russian).

\end{thebibliography}
\end{document}